\documentclass{amsart}
\usepackage{amssymb,enumerate}

\usepackage[colorlinks]{hyperref}

\theoremstyle{plain}

\newtheorem{theorem}{Theorem}[section]
\newtheorem{proposition}[theorem]{Proposition}

\newtheorem{corollary}[theorem]{Corollary}
\newtheorem{conj}[theorem]{Conjecture}

\theoremstyle{definition}

\newtheorem*{ack}{Acknowledgements}

\theoremstyle{remark}
\newtheorem{remark}[theorem]{Remark}

\numberwithin{equation}{section}

\newcommand{\length}{\operatorname{length}}

\newcommand{\bettideg}{\operatorname{\beta-deg}}

\newcommand{\bsx}{\boldsymbol x}
\newcommand{\bbQ}{\mathbb{Q}}
\newcommand{\bbZ}{\mathbb{Z}}
\newcommand{\chr}{\operatorname{char}}
\newcommand{\codim}{\operatorname{codim}}
\newcommand{\Coker}{\operatorname{Coker}}
\newcommand{\cone}{\operatorname{cone}}

\renewcommand{\dim}{\operatorname{dim}\,}
\newcommand{\edim}{{\operatorname{emb\,dim}\,}}
\newcommand{\End}{\operatorname{End}}
\newcommand{\Ext}[4]{\operatorname{Ext}^{#1}_{#2}(#3,#4)}
\newcommand{\fm}{\mathfrak{m}}
\newcommand{\fn}{\mathfrak{n}}
\newcommand{\ges}{\geqslant}
\newcommand{\hh}[1]{\operatorname{H}(#1)}
\newcommand{\HH}[2]{\operatorname{H}_{#1}(#2)}
\newcommand{\CH}[2]{\operatorname{H}^{#1}(#2)}
\newcommand{\Hilb}[2]{\operatorname{Hilb}_{#1}(#2)}
\newcommand{\Hom}{\operatorname{Hom}}
\newcommand{\Ker}{\operatorname{Ker}}
\newcommand{\lra}{\longrightarrow}

\newcommand{\poin}[3]{\operatorname{P}^{#1}_{#2}(#3)}
\newcommand{\rank}{\operatorname{rank}}
\newcommand{\shift}{\mathsf{\Sigma}}
\newcommand{\xra}{\xrightarrow}

\begin{document}

\title[Finite free complexes]{Examples of finite free complexes of \\ small rank and small homology}

\author[Iyengar]{Srikanth B. Iyengar}
\address{Department of Mathematics, University of Utah, Salt Lake City, UT 84112, USA}
\email{iyengar@math.utah.edu}

\author[Walker]{Mark E. Walker}
\address{Department of Mathematics, University of Nebraska, Lincoln, NE 68588, U.S.A.}
\email{mark.walker@unl.edu}

\begin{abstract}
This work concerns finite free complexes over commutative noetherian rings, in particular over group algebras of elementary abelian groups.  The main contribution is the construction of complexes such that the total rank of their underlying free modules, or the total length of their homology, is less than predicted by various conjectures in the theory of transformation groups and in local algebra. 
\end{abstract}

\date{6th May 2018}

\keywords{complete intersection ring, finite free complex, total Betti number, toral rank conjecture}
\subjclass[2010]{13D02 (primary); 13D22, 55M35, 57S17  (secondary)}

\maketitle

\section{Introduction}

In this paper we construct counterexamples to five related conjectures concerning the rank and homology of finite free complexes over
commutative noetherian rings, and, in particular, over group algebras of elementary abelian groups. 

\begin{conj} 
\label{conj:carlsson} 
Let $k$ be a field of positive characteristic $p$, let $E$ be an elementary abelian $p$-group of rank $r$ and $kE$ the corresponding 
group algebra. If $F$ is a bounded complex of free $kE$-modules of finite rank and $\hh F\ne 0$, then $\rank_k \hh F \geq 2^r$.   
\end{conj}

Here $\rank_{k}\hh F$ denotes $\sum_{i}\rank_{k}\HH iF$. Conjecture~\ref{conj:carlsson} is an algebraic generalization of a conjecture in topology due to Carlsson, recalled in Remark~\ref{rem:carlsson}; see~\cite[Question 7.3]{AB}, and ~\cite[Section 2]{AS}.  Carlsson has proved Conjecture~\ref{conj:carlsson} when $p=2$ and $r\le 3$; see~\cite[Theorem~2]{Gc0}.   Corollary~\ref{cor:eapg} below provides a counterexample whenever $p\geq 3$ and $r\ge 8$. 

The next conjecture concerns a graded polynomial ring $R = k[t_1, \dots, t_d]$ over a field $k$, on indeterminates $t_1, \dots, t_d$ of  upper degree two. A \emph{Differential Graded $R$-module} is a graded $R$-module $F$ equipped with an $R$-linear endomorphism $d$ of $F$ that has (upper) degree $1$ and satisfies $d^2 = 0$. Such a DG module  is \emph{semifree} provided there is a chain of graded submodules 
\[
0 = F(-1) \subseteq F(0) \subseteq F(1) \subseteq \cdots \subseteq\, \bigcup_{i\ges 0}F(i) = F
\]
such that $F(i) /F(i-1)$ is a graded free $R$-module and $d(F(i)) \subseteq F(i-1)$ for all $i$. In particular,  ignoring the differential, $F$ itself is a
graded free  $R$-module;  we write $\rank_R F$ for its rank.  For further details, see, for example, \cite[\S1.3]{Av:barca}.

\begin{conj} 
\label{conj:gradedGTRC} 
For $R = k[t_1, \dots, t_d]$ as above, if $F$ is a semifree DG $R$-module such that 
$\rank_{k}\hh F$ is finite and non-zero, then $\rank_{R} F \geq 2^d$.   
\end{conj}

This is a generalization of a topological conjecture due to Halperin; see Remark~\ref{rem:halperin}. For $d\le 3$ Conjecture~\ref{conj:gradedGTRC}  has been proved by Allday and Puppe~\cite[Proposition~1.1 and  Corollary~1.2]{AP2}; see also~\cite[Theorem~5.2, Remark~5.5]{ABI}. Walker~\cite[Theorem~5.3]{Mw0} proved it when $\chr k\ne 2$ and $\hh F$ is concentrated in even degrees or in odd degrees. Corollary~\ref{cor:TRC} below describes counterexamples when $\chr k\ne 2$ and $d\ge 8$; the DG modules constructed have cohomology in degrees $0$ and $3$. 

A  conjecture due to  Buchsbaum and Eisenbud~\cite[Proposition~1.4]{BE}, and Horrocks~\cite[Problem 24]{Rh} predicts that over a local ring $R$ of Krull
dimension $d$,  any free resolution $F$ of a non-zero module of finite length satisfies $\rank_R F_i \geq \binom di$  for all $i$.  In particular, $\rank_R F
\geq 2^d$; this last inequality was conjectured by Avramov, see 
\cite[pp.~63]{EG}, and proved by Walker~\cite[Theorem 1]{Mw} when $R$ is a complete intersection whose residual characteristic is not two, and also when $R$ is any local ring containing a field of positive characteristic not equal to two.

Folklore  has extended Avramov's conjecture to  all finite free complexes.

\begin{conj} 
\label{conj:GTRC} 
If $R$ is a local ring and $F$ is a  complex of  free $R$-modules with $\length_{R}\hh F$ finite and non-zero, then  $\rank_R F  \geq 2^d$, where $d$ is the Krull dimension of $R$.
\end{conj}

For $d\le 3$ this was 
proved  by Avramov, Buchweitz, and Iyengar~\cite[Theorem~5.2]{ABI}. Theorem \ref{thm:multiplicity} below provides counterexamples;
the simplest one occurs  when $R$ is a regular local ring of dimension $8$ and residual characteristic not two.  

The next conjecture concerns differential modules.   A \emph{differential module} over a ring $R$ is an $R$-module $F$ equipped with an $R$-linear endomorphism $d$ satisfying $d^2 = 0$.  For such an $F$, a \emph{free flag} consists of a chain of submodules
\[
0 = F(-1) \subseteq F(0) \subseteq F(1) \subseteq \cdots \subseteq\, \bigcup_{i\ges 0}F(i) = F
\]
such that $F(i)/F(i-1)$ is a free $R$-module and $d(F(i)) \subseteq F(i-1)$ for all $i$. 

\begin{conj}  
\label{conj:DMC} 
If $R$ is a local ring and $F$ is a differential $R$-module that admits a  free flag and has the property that $\length_{R}\hh F$ is finite and non-zero, then $\rank_RF\geq 2^d$, where $d$ is the Krull dimension of $R$. 
\end{conj}

Conjecture  \ref{conj:DMC} is stated in \cite[5.3]{ABI} and proven there for $d \le 3$. Given any chain complex of $R$-modules the direct sum of its components
is a differential $R$-module, called its \emph{compression}. The compression of a  free complex admits  a  free flag. 
Thus Conjecture~\ref{conj:DMC}
implies \ref{conj:GTRC}, so any counterexample to the latter yields one also to the former; see  Corollary~\ref{cor:dm}. 

The final conjecture concerns the sequence of Betti numbers for modules over complete intersection rings.   The \emph{$i$th Betti number} of a finitely
generated module $M$ over a local ring $R$ is the rank of the free $R$-module in degree $i$ in a minimal free resolution of $M$; we denote it by
$\beta^R_i(M)$. A  local ring is a \emph{complete intersection} if its completion is isomorphic to the quotient of a regular local ring by a 
regular sequence. 
Over such a ring the  Betti numbers of any finitely generated $R$-module $M$ are eventually given by a quasi-polynomial of period $2$;
see \cite[Corollary~4.2]{Gu}, also \cite[Theorem 4.1]{Av:Virtual} and \cite[Theorem 7.3]{AvB2}. In detail,  when the projective dimension of $M$ is infinite,
there is a positive integer $c$, called the \emph{complexity} of $M$,  a positive real number $\bettideg M$, called the \emph{Betti degree} of $M$, and
polynomials $q_{ev}$ and $q_{odd}$ of degree at most $c-2$ such for $i \gg 0$ one has 
\[
\beta_i^R(M) = \frac{\bettideg M}{2^{c-1} (c-1)!} i^{c-1} + 
\begin{cases}
q_{ev}(i) & \text{if $i$ is even and } \\
q_{odd}(i) & \text{if $i$ is odd.} \\
\end{cases}
\]
In this notation \cite[Conjecture 7.5]{AvB2} reads as follows.

\begin{conj} 
\label{conj:BDC} 
For any finitely generated module $M$ of complexity $c\ge 1$ over a complete intersection local ring $R$, one has $\bettideg M \geq 2^{c-1}$.
\end{conj}

This conjecture was motivated by a relationship with  Conjecture \ref{conj:GTRC}; 
see \cite[7.4]{AvB2}.
Avramov and Buchweitz~\cite[Remark 7.5.1]{AvB2}  proved this conjecture when $c\le 2$, and it holds  when $q_{ev}=q_{odd}$, in particular when $R$ is the
localization of a graded ring defined by quadrics~\cite[Theorem~2.3]{Av:rational}.

Corollary \ref{cor:multiplicity} provides counterexamples whenever $R$  has
defining relations of order at least three, Krull dimension  $0$, embedding dimension  at least $8$, and residual characteristic not  equal to $2$.

The starting point of the construction of our examples is a result on the existence of Lefschetz elements in exterior algebras, recalled in
Section~\ref{sec:lefschetz}. This connection is present already in the work of Allday and Halperin~\cite{AH}; see also  \cite[Example 4.5]{AP1} by Allday and Puppe, and \cite[Corollary~7.2.5]{FOT}, by F\'elix, Oprea, and Tanr\'e.  

\section{Lefschetz elements in exterior algebras}
\label{sec:lefschetz}
In this section we recall a basic result concerning exterior algebras that underlies all our constructions. The \emph{Hilbert series} of a finite dimensional
$\bbZ$-graded vector space $W=\{W_{i}\}_{i\in\bbZ}$ over a field $k$ is the Laurent polynomial   
\[
\Hilb Wt :=\sum_{i\in\bbZ} \rank_{k}(W_{i})t^{i}   
\]  
with non-negative integer coefficients. Evidently $\Hilb W1 =\rank_{k}W$, and 
\[
\Hilb{\shift^{n}W}t=t^{n}\Hilb Wt\,,
\]
where $\shift^{n}W$ is the graded $k$-vector space with $(\shift^{n}W)_{i}=W_{i-n}$ 
for each $i$.

\begin{proposition}
\label{prp:lefschetz}
Let $k$ be a field and $\Lambda$ the exterior algebra of a $k$-vector space with basis $x_1, \dots, x_n, y_1, \dots, y_n,$ in lower degree one. Thus $\Lambda_i$, the degree $i$ part  of $\Lambda$, is the $i$-th exterior power of the given vector space. 
Set
\[
w:= \sum_{i=1}^{n}x_{i}y_{i}\,   \in \Lambda_2\,,
\]
and let $\lambda_{w}\colon \Lambda \to \shift^{-2}\Lambda$  be the morphism of graded $\Lambda$-modules where $1\mapsto w$.

If $\chr k=0$ or $\chr k > \frac{n+1}{2}$,  then the map
\[
(\lambda_{w})_{i}\colon \Lambda_{i} \to \Lambda_{i+2}
\]
is injective for $i \le n-1$ and surjective for $i \ge n-1$. Moreover, 
we have
\[
\Hilb{\Coker(\lambda_{w})}t = \frac {h(t)}{t^{2}}  \qquad\text{and} \qquad \Hilb{\Ker(\lambda_{w})}t = t^{2n}h(1/t) 
\]
where $h(t):=\sum_{i=0}^{n}\left[\binom {2n}i - \binom{2n}{i-2}\right]t^{i}$, and there is an equality
\[
\rank_k \Coker(\lambda_{w}) + \rank_k \Ker(\lambda_{w})   =  \binom{2n+2}{n+1}\,.
\]
\end{proposition}

\begin{proof} 
See \cite[Proposition~A.2]{CHI} for a proof of the assertion concerning the injectivity/surjectivity of multiplication by $w$. Given this, 
it is elementary to check that the Hilbert series of $\Coker(\lambda_{w})$ and $\Ker(\lambda_{w})$ are as  stated. It remains to note that
\[
\rank_k \Coker(\lambda_{w}) + \rank_k \Ker(\lambda_{w}) = h(1) + h(1) = \binom{2n +2}{n+1}\,,
\]
where the second equality can be verified as follows:
\[
\begin{aligned}
2h(1) 
& =  \sum_{i=0}^{n} \left[\binom{2n}{i}- \binom{2n}{i-2}\right]
\!+\!  \sum_{i=n}^{2n} \left[\binom{2n}i - \binom{2n}{i+2}\right] \\
& = \binom{2n}{ n-1} + 2 \binom{2n}{n} + \binom{2n}{n+1} \\
& = \binom{2n +1}{n} +  \binom{2n+1}{n+1} \\
& = \binom{2n +2}{n+1}.  
\end{aligned}  
\]
This completes the proof.
\end{proof}

\begin{remark} 
\label{rem:lefschetz}
If $\chr k = 0$, then the first assertion in Proposition~\ref{prp:lefschetz} can be proved using the representation theory of $\mathrm{sl}_2(k)$, in a manner
similar to an argument that appears in the proof of the Hard Lefschetz Theorem found in \cite[p.~122]{GH}.  

Indeed, let $x_{1}^{*},\dots,x_{n}^{*},y_{1}^{*},\dots,y_{n}^{*}$ be the basis of $\Hom_{k}(\Lambda_{1},k)$ dual to the  given basis for $\Lambda_{1}$.  
The elements $x_i^*$ and $y_i^*$ 
induce $k$-linear derivations of degree $-1$ 
on $\Lambda$. Set $c:= \sum_i y_i^* x_i^*$; this is an endomorphism of $\Lambda$ of  degree $-2$. The restriction of $h := [c,  \lambda_w]$
to $\Lambda_{j}$ is multiplication by $n - j$. We also have  $[c,h] = -2c$ and $[\lambda_w,h] =2\lambda_w$, and thus the operators $\lambda_w,c,h$ endow 
$\Lambda$ with the structure of
a $sl_2(k)$-representation such that $\Lambda_{j}$ has weight $n-j$.   It follows that  $\lambda_w^j: \Lambda_{n-j} \to  \Lambda_{n+j}$ is an isomorphism for all $j
\geq 1$; see \cite[Chapter IV, Theorem 4(b)]{Se}.  
\end{remark}

\begin{corollary} 
\label{cor:lefschetz}
Let $k$ be a field with $\chr k \ne 2$ and $\Lambda$ an exterior algebra on $k$-vector space of rank $d$. If $d \geq 8$, then there is an element $w \in \Lambda_{2}$  such that
\[
 \rank_k \Coker(\lambda_{w}) + \rank_k \Ker(\lambda_{w})  = 2^{d}-2^{d-6}   < 2^d,
\]
where $\lambda_{w}\colon \Lambda \to \shift^{-2}\Lambda$ is multiplication by $w$. 
\end{corollary}

\begin{proof} 
Let $X$ be a basis for $\Lambda_{1}$. Select any eight element subset  
\[
X' = \{x_1, \dots, x_4, y_1, \dots, y_4\}
\]
of $X$, and set $w = \sum_i x_iy_i$. There is an isomorphism of $k$-algebras $\Lambda = \Lambda' \otimes_k \Lambda''$ where $\Lambda'$ is the algebra generated by $X'$ and $\Lambda''$ is the algebra generated by $X \setminus X'$.  By Proposition \ref{prp:lefschetz} 
\[
\rank_k \Coker(\lambda'_{w})  + \rank_k \Ker(\lambda'_{w}) = 
\binom {10}{5} = 252 = (2^8-4)\,,
\]
where $\lambda'_{w}\colon \Lambda' \to \shift^{-2}\Lambda'$ is, as before, multiplication by $w$. There are  isomorphisms of $k$-vector spaces
\begin{align*}
\Coker(\lambda_{w}) & \cong \Coker(\lambda'_{w}) \otimes_k \Lambda''\\
\Ker(\lambda_{w}) &\cong \Ker(\lambda'_{w}) \otimes_k \Lambda'' 
\end{align*}
from which we deduce
\[
\rank_k \Coker(\lambda_{w})   + \rank_k \Ker(\lambda_{w}) =     2^d - 2^{d-6}.   \qedhere
\]
\end{proof}

\begin{remark} 
If $\chr k = 2$, then for every $w\in \Lambda_{\ges 1}$ we have $w^2 = 0$ 
and hence
\[
\rank_k \Coker(\lambda_{w}) + \rank_k \Ker(\lambda_{w})    \geq 2^d.
\]
Thus  Corollary~\ref{cor:lefschetz} does not extend to $\chr k=2$; this is why all our examples are in characteristic not equal to two.
\end{remark}

\begin{remark} \label{E129}
The numbers $\binom{2n+2}{n+1}$, called \emph{central binomial coefficients}, are related to the  Catalan numbers,  $\{C_{n}\}$, by the formula
$\binom{2n+2}{n+1} = (n+2) C_{n+1}$.   Our counter-examples involve values of $n$ for which the inequality 
\begin{equation} 
\label{E531}
\binom{2n+2}{n+1} < 2^{2n}
\end{equation}
holds. As seen in the proof of Corollary \ref{cor:lefschetz}, it holds when $n = 4$, and this is the smallest value of $n$ for which it does. 
It follows from Stirling's 
formula that
\[
\binom{2n+2}{ n+1}  < 2^{2n} \cdot \frac{4}{\sqrt{\pi(n+1)}} 
\]
for all $n \geq 0$; see \cite[1.5]{Bb}. 
In particular, \eqref{E531} holds for all $n \geq 5$ too.
\end{remark}

\begin{remark}
\label{rem:upper} 
In Section~\ref{sec:rest} we need versions of  Proposition~\ref{prp:lefschetz} and Corollary~\ref{cor:lefschetz}   in which the $x_i$'s and $y_i$'s have lower
degree $-1$. 
In this case, we switch to upper indexing: by convention, for a graded object $X$, the component of upper degree $i$, written $X^i$, is defined to be $X_{-i}$.
Note that $(\shift^n X)^i = X^{n+i}$. We define the Hilbert series of a  graded vector space $W=\{W^{i}\}_{i\in\bbZ}$ satisfying $W^i = 0$ for $i \ll 0$  to be
$\Hilb{W}t = \sum_i \rank_k(W^i) t^i$.  

When $x_i, y_i\in \Lambda^1$ for all $i$,  the map in Proposition~\ref{prp:lefschetz}   takes the form
\[
\lambda_{w}\colon \Lambda \lra \shift^{2}\Lambda
\]
The Hilbert series of the cokernel and kernel of this morphism are still the same:
\[
\Hilb{\Coker(\lambda_{w})}t = \frac {h(t)}{t^{2}}  \qquad\text{and} \qquad \Hilb{\Ker(\lambda_{w})}t = t^{2n}h(1/t) 
\]
where $h(t)$ is as in Proposition~\ref{prp:lefschetz}. The equation in Corollary~~\ref{cor:lefschetz} remains valid.
\end{remark}

\section{Homology of finite free complexes}
\label{sec:homology}
In this section we construct counterexamples to Conjecture~\ref{conj:carlsson}. 

Let $(R,\fm,k)$ be a (commutative, noetherian) local ring $R$, with maximal ideal $\fm$ and residue field $k$. 
The \emph{embedding dimension} of $R$ is the integer 
$$
\edim R = \rank_k (\fm/\fm^2)
$$ 
and the \emph{codimension} of $R$ is the integer 
$$
\codim R = \edim R - \dim R.
$$
The $\fm$-adic completion of $R$ has the form $Q/I$, where  $(Q,\fn,k)$ is a regular local ring and 
$I \subseteq \fn^{2}$; see \cite[\S2.3]{BH}. For any such presentation, we have
$$
\edim R = \dim Q \quad \text{ and } \quad
\codim R = \dim Q - \dim R.
$$
We say $R$ is a \emph{complete intersection} if $\codim R =\rank_{k}(I/\fn I)$ or, equivalently, if  
$I$ can be generated by a $Q$-regular sequence; see \cite[Theorem~2.3.3]{BH}.

In the sequel, given a complex $X$ of $R$-modules with differential $d^{X}$ and an integer $m$, the shifted graded module $\shift^m X$ 
is a complex with differential $d^{\shift^mX} = (-1)^m  d^X$.
A \emph{finite free complex} of $R$-modules is a complex of the form
\[
0 \lra F_b \lra \cdots \lra F_a \lra 0 
\]
with each $F_i$ free of finite rank.

\begin{theorem}
\label{thm:homology}
Let $(R,\fm,k)$ be a complete intersection of codimension $r$.  If $r\geq 8$ and $\chr k \ne 2$, then there is a finite free complex of $R$-modules  $F$ with
\[
 \length_R \hh F =   2^r - 2^{r-6}.
\]
\end{theorem}

\begin{proof}
Let $K$ the Koszul complex on a minimal set 
of generators of $\fm$. Then $K$ is a commutative DG $R$-algebra, $\HH 1K$ is a $k$-vector space of dimension $r$, and there is an isomorphism of graded $k$-algebras
\begin{equation}
\label{eq:hK}
\hh K \cong \Lambda_{k} (\HH 1K); 
\end{equation}
see, for instance, \cite[Theorem~2.3.1]{BH}.  Set $\Lambda = H(F)$, 
let $w \in \HH 2K$ be an element as in Corollary \ref{cor:lefschetz}, and let $z \in K_{2}$ be a cycle representing $w$. 
Since $K$ is a DG algebra, multiplication by $z$ determines a morphism of DG $K$-modules  
\begin{equation}
\label{eq:kappa}
\lambda_z \colon  K \lra \shift^{-2} K \quad\text{given by  $u\mapsto uz$.}
\end{equation}
Set $F:= \cone(\lambda_z)$, the mapping cone of the morphism $\lambda_z$. 
There is an exact sequence of DG $K$-modules
\[
0\lra  \shift^{-2} K \lra F \lra \shift K \lra 0.
\]
Since $K$  is a finite free $R$-complex so is $F$.  The associated  exact sequence in homology has the form   
\[
\cdots \lra \HH jK \xra{\ \lambda_w\ } \HH {j+2}K \lra \HH jF \lra \HH {j-1}K \xra{\ \lambda_w\ } \HH{j+1}K  \lra \cdots
\]
Thus there is an exact sequence of graded $R$-modules 
\[
0 \to \Coker\left(\hh K \xra{\lambda_w} \shift^{-2} \hh K \right) \to  \hh F \to \Ker\left(\shift \hh K \xra{\lambda_w} \shift^{-1} \hh K \right) \to 0\,.
\]
It follows that
\[
\length_R \hh F = \rank_k  \Coker(\lambda_{w}) + \rank_k \Ker(\lambda_{w})  = 2^r - 2^{r-6},
\]
where the second equality is by the choice of $w$; see Corollary~\ref{cor:lefschetz}.
\end{proof}

\begin{corollary} 
\label{cor:eapg}
Let $p$ be an odd prime, $k$ a field of characteristic $p$, and $E$ an elementary abelian $p$-group of rank $r$.  If $r \geq 8$, there is a finite complex $F$ of free $kE$-modules such that $\rank_k \hh F < 2^r$. Thus  Conjecture \ref{conj:carlsson} fails when $p$ is odd. 
\end{corollary}

\begin{proof} 
There is an isomorphism of $k$-algebras
\[
kE \cong \frac{k[t_1, \dots, t_r]}{(t_1^p, \dots, t_r^p)}
\]
so that $kE$ is a complete intersection of codimension $r$. The result thus follows from Theorem~\ref{thm:homology}, since $\rank_kM =
\length_{kE}M$ for every $kE$-module $M$.
\end{proof}

\begin{remark}
\label{rem:carlsson}
Conjecture \ref{conj:carlsson}  is extrapolated from a conjecture of Carlsson~\cite[pp.~333]{Gc0}, also \cite[I.3]{Gc},  predicting that if a finite CW complex
$X$ admits a free, cellular $E$-action, then the total rank of its singular homology with $\bbZ/p$-coefficients, $\HH{*}{X, \bbZ/p}$, is at least $2^r$. In this
situation, $\HH{*}{X, \bbZ/p}$ is realized as the homology of a complex $F$ of $kE$-modules satisfying the hypotheses of Conjecture \ref{conj:carlsson} with $k =
\bbZ/p$.  Thus, Conjecture \ref{conj:carlsson} implies Carlsson's Conjecture, but we do not know whether the complex $F$ in Corollary~\ref{cor:eapg} arises from
a space with a free $E$-action. 

\end{remark}

\section{Total rank and Betti degree of complexes}
\label{sec:rest}
In this section we construct counterexamples to Conjectures~\ref{conj:gradedGTRC}--\ref{conj:BDC}.

For any local ring $(R,\fm,k)$ one has an inequality
\[
\length_R(R/\fm^3) \leq {\binom{\edim R +1}2}\,.
\]
When equality holds we say that the defining relations of $R$ are \emph{of order at least three}. This is equivalent to the condition that in some presentation of the $\fm$-adic completion of $R$ as  $Q/I$, for a regular local ring $(Q,\fn,k)$, one has $I \subseteq \fn^3$.

Henceforth $(R, \fm, k)$ will be a complete intersection; see the start of Section~\ref{sec:homology} for the meaning. Let $M$ be an $R$-complex with the $R$-module $\hh M$ finitely generated. As for modules, the Betti numbers $\{\beta^{R}_{i}(M)\}_{i\in\bbZ}$ of $M$ are the ranks of the free modules in a minimal free resolution of $M$, see \cite[\S1.1]{Pr}, and can be computed as
\[
\beta^{R}_{i}(M) = \rank_{k}\Ext iRMk\,.
\]
These numbers are finite for all $i$ and are equal to zero for $i\ll 0$. The \emph{Poincar\'e series} of $M$ is the generating series 
\[
\poin RMt:= \sum_{i\in \bbZ} \beta^{R}_{i}(M)t^{i} \in \bbZ[|t|][t^{-1}].
\]
There exist an integer $0\le c\le \codim R$ and a Laurent polynomial $p_M(t)$ with integer coefficients satisfying $p_{M}(1)\ne 0$, such that
\begin{equation}
\label{eq:poincare}
  \poin RMt = \frac {p_{M}(t)}{(1-t^{2})^{c}}\,.
\end{equation}
This result is due to Gulliksen~\cite[Corollary~4.2]{Gu}; see also \cite[Theorem 9.2.1]{Av:barca}.  The integer $c$ is the \emph{complexity} of $M$; Remark~\ref{rem:Betti-degree} reconciles this definition with the one given in the Introduction.

We are interested in the integer $p_{M}(1)$. If $c = 0$, then  
\[
p_M(1) = \poin RM1 = \sum_i \beta_i^R(M),
\]
the total Betti number of $M$. In view of this,  when $\codim R=0$ the next result  provides counterexamples to Conjecture \ref{conj:GTRC}.

\begin{theorem}
\label{thm:multiplicity}
Let $(R, \fm, k)$ be a complete intersection whose defining relations have order at least $3$.  If $\chr k\ne 2$ and $e := \edim R$ is at least $8$,  then there exists a complex $F$ with $\HH 0F \cong k \cong \HH 1F$ and $\HH iF = 0$ for all $i \ne 0,1$, with the property that
\[
p_{F}(1) =        2^e - 2^{e-6}.
\]
Moreover, when $\codim R\ge 1$ there exists a finitely generated $R$-module $M$ with
\[
 p_M(1) = 2^e - 2^{e-6}.
\]
\end{theorem}

\begin{proof}
Set $c=\codim R$. Since the defining relations of $R$ have order at least $3$, there is an isomorphism of $k$-algebras
\begin{equation} 
\label{E615}
\Ext {}Rkk \cong \Lambda \otimes_k S
\end{equation}
where $\Lambda$ is an exterior algebra generated by $e$ elements of upper degree one, and  $S$ is a polynomial algebra generated by $c$  elements of upper degree two; see \cite[Example~10.2.3]{Av:barca}.  Choose $w\in \Lambda^2$ as in Corollary~\ref{cor:lefschetz}; see also Remark~\ref{rem:upper}. Viewed as an element in $\Ext{}Rkk$, the element $w\otimes 1$ represents a morphism of $R$-complexes
\[
\zeta \colon \shift^{-2}X \lra X
\]
where $X$ is a minimal $R$-free resolution of $k$. Let $F=\shift \cone(\zeta)$, so there is an exact sequence of $R$-complexes
\[
0\lra \shift X \lra F\lra X \lra 0\,.
\]
The induced exact sequence in homology 
\[
\cdots \lra \HH{i+1}X \lra \HH{i-1}X \lra \HH iF \lra \HH iX \lra \HH{i-2}X \lra \cdots
\]
gives $\HH 0F \cong k\cong \HH 1F$ and $\HH iF = 0$ for all $i \ne 0,1$.  

Under the isomorphism \eqref{E615}, the endomorphism $\Ext {}R{\zeta}k$ of $\Ext {}Rkk$ induced by  $\zeta$ corresponds 
to the map 
\[
\lambda_{w} \otimes 1\colon \Lambda \otimes_k S \to \shift^2 \Lambda \otimes_k S,
\]
and thus there is an exact sequence of $\Lambda \otimes_k S$-modules
\[
0 \lra \shift^{-2} \Coker(\lambda_{w}) \otimes_k S \lra \Ext {}RFk \lra \shift^{-1} \Ker(\lambda_{w}) \otimes_k S \lra 0\,.
\]
As a sequence of graded $S$-modules, this sequence splits and  yields an isomorphism 
\[
\Ext {}RFk \cong W \otimes_k S
\]
of graded $S$-modules,
where $W$ is the graded $k$-vector space $\shift^{-2} \Coker(\lambda_{w}) \oplus \shift^{-1} \Ker(\lambda_{w})$. Since  the generating series of $S$ is $1/(1-t^2)^c$,  we  have  
\begin{equation}
\label{eq:poinF}
\poin RFt = \frac{\sum_{i} \rank_{k}(W^{i}) t^i}{(1-t^2)^c}\,.
\end{equation}
Evaluated at $t=1$, the numerator equals $\rank_{k}W$, which is non-zero because $W$ is non-zero. This justifies the first equality below; the second one is from Corollary~\ref{cor:lefschetz}:  
\[
p_F(1) = \rank_k W =  \rank_k \Coker(\lambda_{w}) + \rank_k \Ker(\lambda_{w})  = 2^{e-8} (2^{8}-4)
\]
This proves the first assertion.

Assume $c\ge 1$, so that $R$ is not regular. From \eqref{eq:poinF} it follows that the complexity of $F$ equals $c$ and that
\[
p_{F}(t) = \sum_{i} \rank_{k}(W^{i}) t^i
\]
Let $G$ be a minimal free resolution of $F$ and set $M:=\Coker (G_{2}\to G_{1})$. Since $\HH iG =\HH iF = 0$ for $i \geq 2$, the complex $\shift^{-1}(G_{\ges
  1})$ is a minimal free resolution of $M$, and hence  
\begin{equation}
\label{eq:poinM}
\poin RMt = \frac{\poin RFt - \poin RF0}{t} = \frac{\poin RFt - 1}{t} = \frac{(p_F(t)- (1-t^{2})^{c})t^{-1}}{(1-t^{2})^{c}}.
\end{equation}
This implies that the complexity of $M$ is also $c$, and since $c\ge 1$ this yields the first equality below:
\[
p_{M}(1)=p_{F}(1) = \rank_k W =2^{e}-2^{e-6}\,.
\]
The remaining equalities have already been justified. 
\end{proof}

\begin{remark}
In the course of the proof of the preceding result, we have in fact calculated the Poincar\'e series of the complex $F$. It is 
\[
 \poin RFt  = \frac{(1+t)^{e-8} (1+8t+27t^{2}+48t^{3}+42t^{4}+42t^{5}+48t^{6}+27t^{7}+8t^{8}+t^{9}) }{(1-t^{2})^{c}}   
\]
Using \eqref{eq:poinM} one can also compute the Poincar\'e series of $M$. 
\end{remark}

\begin{remark}
\label{rem:Betti-degree}
Let $R$ be a complete intersection, $M$ an $R$-complex with $\hh M$ finitely generated, and $c$ its complexity. If $c \geq 1$,  then from \eqref{eq:poincare} one gets that there are polynomials $q_{ev}$ and $q_{odd}$ of degree at most $c-2$ such for $i \gg 0$ one has
\[
\beta_i^R(M) = \frac{\bettideg M}{2^{c-1} (c-1)!} i^{c-1}+ \begin{cases}
q_{ev}(i) & \text{if $i$ is even and } \\
q_{odd}(i) & \text{if $i$ is odd. } \\
\end{cases}
\]
See also \cite[7.3]{AvB2}. It follows that the coefficient of $t^i$ in  $(1-t^2)^{c-1} \poin RMt$ is $\bettideg M$ for $i \gg 0$; that is to say, there are equalities
\[
\frac{p_M(t)}{1-t^2} = (1-t^2)^{c-1} \poin RMt = \frac{\bettideg M}{1-t} + l(t)
\]
for some Laurent polynomial $l(t)$.  In particular there is an equality
\[
p_M(1) = 2 \bettideg M\,.
\]
In view of this equality, when $\dim R=0$, that is to say, when $\edim R = \codim R$, Theorem~\ref{thm:multiplicity} specializes to the following statement.
\end{remark}

\begin{corollary} 
\label{cor:multiplicity}
Let $R$ be a complete intersection with defining relations of order at least $3$ and $\dim R=0$. If $\chr k\ne 2$ and $c:=\codim R\ge 8$, then there exists a finitely generated $R$-module $M$ with $\bettideg M < 2^{c-1}$. Thus Conjecture \ref{conj:BDC} fails. \qed
\end{corollary}

\begin{remark}
\label{rem:poin-F}
Let $n$ be a positive integer, $(R,\fm,k)$ a regular local ring of  dimension $2n$, and assume $\chr k > (n+1)/2$. Then $\Ext{}Rkk$ is the exterior algebra on a $k$-vector space of rank $2n$. Let $w$ be as in Proposition~\ref{prp:lefschetz} and $G$ the complex constructed from $w$ as in the proof of Theorem~\ref{thm:multiplicity} above. A direct computation using Proposition~\ref{prp:lefschetz} yields 
\[
\poin RGt = h(t) + t^{2n+1}h(1/t)\,.
\]
Hence the sequence of Betti numbers of $G$ is palindromic. By construction of $G$, the module $M:=\Coker(G_{2}\to G_{1})$ fits into an extension
\[
0\lra k\lra M\lra \fm \lra 0\,.
\]
The projective dimension of $M$ equals $2n$. The $R$-module $M$ is locally free on the punctured spectrum, and hence the same is true of its syzygy modules, $\Omega^{d}_{R}(M)$.  The Poincar\'e series of $M$ is $(\poin RGt-1)/t$, see \eqref{eq:poinM}, so the ranks of its  syzygy modules are 
\[
\rank_{R}\Omega^{d}_{R}(M)=
\begin{cases}
\binom {2n}d - \binom {2n}{d-1} & \text{for $0\le d\le n$} \\
\binom {2n}d - \binom {2n}{d+1}  & \text{for $n+1\le d\le 2n$.} 
\end{cases}
\]
The projective dimension of $\Omega^{d}_{R}(M)$ is $2n-d$ and its depth is $d$. This computation has a bearing on \cite[Question 25]{Rh}. Indeed, fix $1\le d\le 2n-1$ and set
\[
s(d):= \min\left\{\rank_{R}L\left|
\begin{aligned} 
\text{$L$ is finitely generated, of depth $d$, and is}\\
\text{ locally free on the punctured spectrum}
\end{aligned}
\right.\right\}
\]
For $d=n$, the computation above yields 
\[
s(n)\le \rank_{R}\Omega^{n}_{R}(M)= \binom{2n}{n}-\binom{2n}{n-1}\,.
\]
This is much better than the bound $s(n)\le \binom {2n-1}{n-1}$ given by the $n$th syzygy of $k$.
\end{remark}

\begin{corollary} 
\label{cor:dm}
Let $(R, \fm, k)$ be a regular local ring of dimension $d$. If $\chr k\ne 2$ and $d \geq 8$, then there is a differential $R$-module $D$ such that $\length_{R}\hh D$ is non-zero and finite,  $D$ admits a  free flag and  $\rank_RD <  2^{d}$. Thus Conjecture \ref{conj:DMC} fails.
\end{corollary}

\begin{proof}
Let $D$ be the compression (see the Introduction)  
of a minimal resolution $G$ of the complex $F$ constructed in Theorem~\ref{thm:multiplicity}.  Then $D$ is a differential $R$-module with homology  $\hh G \cong \hh F\cong k^{2}$; in particular the homology of $D$ is non-zero and of finite length. Moreover, $D$ has a  free flag because it is the compression of a free complex; see \cite[2.8(6)]{ABI}.  The minimality of $G$ gives $\rank_{R}D = \rank_R(G) = \sum_{i}\beta^{R}_{i}(F) < 2^d$. 
\end{proof}

\begin{corollary} 
\label{cor:TRC}
Let $R:=k[t_{1},\dots,t_{d}]$ be the polynomial ring over a field $k$ in indeterminates $t_{1},\dots, t_{d}$ of upper degree two, viewed as a DG algebra with zero differential.  If $\chr k \ne 2$ and $d \geq 8$, then there is a semifree DG $R$-module  $G$ with  $\CH{0}{G}\cong k \cong \CH{3}{G}$ and $\CH{j}{G} = 0$ for all $j \ne 0, 3$  and such that  $\rank_R G < 2^{d}$. Thus Conjecture \ref{conj:gradedGTRC} fails.
\end{corollary} 

\begin{proof} 
We construct $G$ by mimicking the argument of Theorem \ref{thm:multiplicity} in the setting of DG-modules. In detail, let $X$ be the Koszul resolution of $k$,
given by the commutative DG-$R$-algebra generated by elements $e_1, \dots, e_d$ of upper 
degree one and $d(e_i) = t_i$. 
Since $X$ is quasi-isomorphic to $k$ as DG-$R$-modules,  $\hh{\End_R(X)}$ is an exterior algebra on $d$ upper degree $-1$ elements. 
Let  $\zeta: \shift^2 X \to X$ be a degree $-2$ cycle in $\End_R(X)$ that represents the degree $-2$ element $w$ of the exterior algebra $\hh{\End_R(X)}$ given
by  Corollary \ref{cor:lefschetz}, and define  $F = \shift^{-3} \cone(\zeta)$,  so that there is an exact sequence 
\begin{equation} 
\label{eq:E116}
0 \lra \shift^{-3} X \lra F \lra  X \lra 0
\end{equation}
of DG-$R$-modules. It follows that $\CH{3}F \cong k \cong \CH{0}F$ and $\CH{i}F = 0$ for $i \ne 0,3$.  Now take $G$ to be a minimal DG-$R$-module associated to
$F$. Evidently $G$ has the same cohomology as $F$ and, since it is minimal, $\rank_RG = \rank_k(G \otimes_R k)$. The exact sequence \eqref{eq:E116} induces the
exact sequence
\[
\begin{aligned}
\cdots \lra & \CH{j+2}{\End_R(X)} \xra{\ \lambda_w \ } \CH{j}{\End_R(X)} \lra \CH{j}{\Hom_R(F, X)}\\
& \lra \CH{j+3}{\End_R(X)} \xra{\ \lambda_w \ }  \CH{j+1}{\End_R(X)} \lra  \cdots 
\end{aligned}
\]
Given this, Corollary \ref{cor:lefschetz} yields the inequality below
\[
\rank_k(G \otimes_R k) = \rank_k\hh{F \otimes_R k} = \rank_{k}\hh{\Hom_{R}(F,X)}< 2^d\,.
\]
The first equality holds because the DG $R$-modules $F$ and $G$ are quasi-isomorphic and semifree; the second one is by adjunction, as $X$ and
$k$ are quasi-isomorphic. 
\end{proof}

\begin{remark}
\label{rem:halperin}
Halperin's Toral Rank Conjecture~\cite[Problem 1.4]{Sh} predicts that  for any topological space $X$ that is reasonable
(say, a finite nilpotent CW complex) and 
 that admits a free action of a $d$-dimensional torus $T$, the rational homology of $X$ satisfies 
\[
\sum_i \rank_{\bbQ}\HH{i}{X, \bbQ}\geq 2^{d}\,.
\]
The validity of 
Conjecture \ref{conj:gradedGTRC} implies the Toral Rank Conjecture: If $X$ admits such an action, then the relative minimal model of the corresponding Borel
fibration is a semifree DG module $F$ over $R:=\bbQ[t_{1},\dots,t_{d}]$ with $\rank_{\bbQ}\hh F$ finite and non-zero,  and $\rank_{\bbQ}\CH{i}{\bbQ\otimes_{R}F}
=\rank_{\bbQ}\CH{i}{X, \bbQ}$;   
see~\cite[\S6]{AH} or \cite[\S7.3.2]{FOT}. Then Conjecture~\ref{conj:gradedGTRC} applied to a minimal DG $R$-module
quasi-isomorphic to $F$ would yield the desired lower bound on $\sum_i \rank_{\bbQ} \HH{i}{X, \bbQ}$. 
However, our counterexamples do not affect the status of the Toral
Rank Conjecture because the complex $G$ in Corollary~\ref{cor:TRC} 
{\em cannot} 
be quasi-isomorphic, even as a DG $R$-module, to any $F$ that arises as above. 

Indeed such an $F$ would come equipped with a morphism of of DG algebras $\phi\colon R\to F$, and since $\CH 2F\cong \CH 2G=0$, by construction, a standard argument in the homotopy theory of DG algebras, see \cite[\S2.2]{FOT}, implies that $\phi$ is homotopic to morphism that factors through $\bbQ$, and hence that 
\[
\rank_{\bbQ}\hh{\bbQ\otimes_{R}F}\geq \rank_{\bbQ}\mathrm{Tor}^{R}(\bbQ,\bbQ) = 2^{d}\,.
\]
This implies $\rank_{R}G\ge 2^{d}$, contradicting the conclusion of Corollary~\ref{cor:TRC}.
\end{remark}

\begin{remark} 
\label{rem:new} 
In characteristic $0$, Conjectures \ref{conj:gradedGTRC}--\ref{conj:BDC} admit families of counter-examples in which the value of the appropriate invariants deviate from the predicted one in an increasing fashion. 

For example, for each $n \geq 1$, if $R$ is any regular local ring of dimension $2n$ whose residue field $k$ has characteristic $0$, then the construction in the proof of Theorem \ref{thm:multiplicity}  gives a minimal finite free complex $G$ such that  $\HH 0G \cong k \cong \HH 1G$ and $\HH iG = 0$ for all $i \ne 0,1$. Moreover, by Remark \ref{E129} one has
\[
\rank_R(G) = \binom{2n+2}{n+1} < 2^{2n} \cdot \frac{4}{\sqrt{\pi(n+1)}}.
\]
The difference $2^{2n} - \binom{2n+2}{n+1}$ tends to $\infty$ as $n$ goes to $\infty$, but  ${\binom{2n+2}{n+1}}^{1/2n}$ tends to $2$. This suggests a question: 

\emph{Is there a real number $a>1$ such that each finite free complex $F$ of modules over a regular local ring $R$ with $\hh F$ non-zero and of finite length satisfies
\[
\rank_R(F) \geq a^{\dim R}?
\]
}
The family of examples constructed here shows that such an $a$ must satisfy 
\[
a \leq \min \left\{\binom{2n+2}{n+1}^{\frac{1}{2n}} \right\} < 1.9605.
\]
\end{remark}

\begin{remark} 
\label{rem:av}
Let $R$ be a regular local ring of dimension $d \geq 8$ and $F$ the complex in Theorem~\ref{thm:multiplicity}. As $\fm \hh F=0$ one has that $\hh F$ is a module over $R/\fm$ and hence also over $R/(\bsx)$, where $\bsx$ is any system of parameters for $R$. Since $R$ is regular, $\bsx$ is a regular sequence and the Koszul complex, say $E$, on $\bsx$ is a $R$-free resolution of $R/(\bsx)$. However,  there cannot be a DG $E$-module structure on $F$: If there were, then $\rank_{R}F\ge 2^{d}$ by \cite[Theorem~5.1]{AIN}, contrary to the conclusion of Theorem~\ref{thm:multiplicity}.  See also \cite[5.4]{AIN}.
\end{remark}

\begin{ack}
It is a pleasure to thank Lucho Avramov, Dave Benson, and Seth Lindokken for helpful conversations, which extended the scope of this work.  We are grateful to Volker Puppe and Matthias Franz for comments and suggestions. We are also indebted to one of the referees who gave us extensive feedback on the first version of this manuscript, leading to substantial revisions. SBI was partially supported by NSF grant DMS-1700985  and MEW was partially supported by  grant  \#318705  from the Simons Foundation.
\end{ack}


\end{document}